\documentclass[tikz]{nmd-mod/article}
\ifnmd
\fi
\usepackage{enumerate}

\title{Norms on the cohomology \\ of hyperbolic 3-manifolds}

\author{Jeffrey F. Brock}
\givenname{Jeffrey}
\surname{Brock}
\address{ Dept.~of Math.\\ 
Brown University \\ 
Box 1917 \\ 
Providence, RI 02912, USA
}
\email{jeff\_brock@brown.edu}
\urladdr{http://www.math.brown.edu/~brock/}

\author{Nathan M. Dunfield}
\givenname{Nathan}
\surname{Dunfield}
\address{ Dept.~of Math., MC-382 \\
          University of Illinois \\
          1409 W. Green St. \\
          Urbana, IL 61801, USA
}
\email{nathan@dunfield.info}
\urladdr{http://dunfield.info}

\arxivreference{}
\arxivpassword{}

\subjclass[2000]{} % Really these are 2010 MSC numbers
\newcommand{\thnorm}[1]{\norm{#1}_{\mathit{Th}}}
\newcommand{\dualthnorm}[1]{\norm{#1}^*_{\mathit{Th}}}
\newcommand{\harnorm}[1]{\norm{#1}_{L^2}}
\newcommand{\hnorm}[1]{\norm{#1}_{h}}
\newcommand{\harnormB}[1]{\norm{#1}_{L^2(B)}}
\newcommand{\onenorm}[1]{\norm{#1}_{L^1}}
\newcommand{\areanorm}[1]{\norm{#1}_{\mathit{LA}}}
\newcommand{\infnorm}[1]{\norm{#1}_{L^\infty}}

\DeclareMathOperator{\interior}{int}
\DeclareMathOperator{\image}{image}
\DeclareMathOperator{\length}{length}
\newcommand{\Mthin}{M_{\mathit{thin}}}
\newcommand{\Mthick}{M_{\mathit{thick}}}
\newcommand{\dVol}{\mathit{dVol}}
\newcommand{\dA}{\mathit{dA}}
\newcommand{\dr}{\mathit{dr}}
\newcommand{\ds}{\mathit{ds}}
\newcommand{\df}{\mathit{df}}
\newcommand{\dz}{\mathit{dz}}
\newcommand{\dt}{\mathit{dt}}

\newcommand{\dg}{\mathit{dg}}
\renewcommand{\dh}{\mathit{dh}}
\newcommand{\dtheta}{\mathit{d}\theta}

\newcommand{\dalphabar}{\mathit{d}\alphabar}
\newcommand{\dphi}{\mathit{d}\phi}
\newcommand{\twoFone}[4]{{}_2 F_1 \left(#1, #2; \, #3;\, #4 \right)}
\newcommand{\thinM}{M_{\mathit{thin}}}
\newcommand{\thickM}{M_{\mathit{thick}}}
\newcommand{\Ylm}{Y_{\ell m}}
\newcommand{\psilm}{\Psi_{\ell m}}
\newcommand{\psil}{\psi_\ell}
\newcommand{\alm}{a_{\ell m}}
\newcommand{\omegalm}{\omega_{\ell m}}
\newcommand{\drhat}{\kernhat{1}{\dr}{-0.5}}
\newcommand{\dphihat}{\kernhat{1}{\dphi}{-0.5}}
\newcommand{\dthetahat}{\kernhat{1.5}{\dtheta}{-0.5}}
\newcommand{\starbar}{\overline{*}}
\newcommand{\nuformula}[1]{6 \pi \left( #1 + 2 #1 \csch^2(#1) - \coth(#1)\left(#1^2 \csch^2(#1) + 1\right)\right)}

\begin{document}

\begin{abstract} 
  We study the relationship between two norms on the first cohomology
  of a hyperbolic 3-manifold: the purely topological Thurston norm and
  the more geometric harmonic norm.  Refining recent results of
  Bergeron, \Sengun, and Venkatesh as well as older work of
  Kronheimer and Mrowka, we show that these norms are
  roughly proportional with explicit constants depending only on the
  volume and injectivity radius of the hyperbolic 3-manifold itself.
  Moreover, we give families of examples showing that some (but not
  all) qualitative aspects of our estimates are sharp.  Finally, we
  exhibit closed hyperbolic \3-manifolds where the Thurston norm grows
  exponentially in terms of the volume and yet there is a uniform
  lower bound on the injectivity radius.
\end{abstract}
\maketitle

\section{Introduction}

Suppose $M$ is a closed oriented hyperbolic \3-manifold.  Our main
goal here is to understand the relationship between two norms on
$H^1(M; \R)$: the purely topological Thurston norm and the more
geometric harmonic norm. Precise definitions of these norms are given
in Section~\ref{sec:norms}, but for now here is an informal sketch.  

The \emph{Thurston norm} $\thnorm{\phi}$ of an integral class
$\phi \in H^1(M; \R)$ measures the topological complexity of the
simplest surface dual to $\phi$; it extends to all of $H^1(M;\R)$
where its unit ball is a finite-sided polytope with rational vertices.
It makes sense for any \3-manifold, though unlike in the hyperbolic
case where it is nondegenerate, there can be nontrivial $\phi$ of norm
$0$.  While it was introduced by Thurston in the 1970s
\cite{Thurston1986}, its roots go back to the early days of topology,
to questions about the genus of knots in the \3-sphere, and it has
been extensively studied in many contexts.

Turning to geometry, as with any Riemannian manifold, the hyperbolic
metric on $M$ gives a norm on $H^1(M; \R)$ which appears in the
proof of the Hodge theorem. Specifically, if we identify $H^1(M; \R)$
with the space of harmonic \1-forms, then the \emph{harmonic norm}
$\harnorm{\cdotspaced}$ is the one associated with the usual inner
product on the level of forms:
\[
\pair{\alpha, \beta} = \int_M \alpha \wedge *\beta
\]
As it comes from a positive-definite inner product, the unit ball of 
$\harnorm{\cdotspaced}$ is a nice smooth ellipsoid.  The harmonic norm
appears, for example, in the Cheeger-M\"uller formula for the
Ray-Singer analytic torsion of $M$ \cite{Cheeger1979, Muller1978}.

By Mostow rigidity, the hyperbolic metric on $M$ is unique, and so a
posteriori the harmonic norm depends solely on the underlying topology
of $M$.  It is thus very natural to ask how these two norms are
related.  Kronheimer and Mrowka seem to have been the first to study
this question in the more general context of arbitrary Riemannian
metrics on $M$; specifically, using deep results from gauge theory,
they characterized the Thurston norm as the infimum (over all
possible metrics) of certain scaled harmonic norms
\cite{KronheimerMrowka1997b}.  Their results have specific
consequences for any given metric including the hyperbolic one; see
Section~\ref{sec:gauge} for a complete discussion.  More recently, Bergeron,
\Sengun, and Venkatesh \cite{BergeronSengunVenkatesh2016} examined
the relationship between these two norms in the case of the hyperbolic
metric.   There,
motivated by questions about torsion growth in the homology of
arithmetic groups, they proved the following beautiful result:
\begin{theorem}[{\cite[4.1]{BergeronSengunVenkatesh2016}}]\label{prop:BSV}
  Suppose $M_0$ is a closed orientable hyperbolic \3-manifold.  There
  exist constants $C_1$ and $C_2$, depending on $M_0$, so that for
  every finite cover $M$ of $M_0$ one has
  \begin{equation}\label{eq:norms}
  \frac{C_1}{\vol(M)} \thnorm{\cdotspaced} \leq \harnorm{\cdotspaced}
  \leq C_2 \thnorm{\cdotspaced} \mtext{on $H^1(M ; \R)$.}
  \end{equation}
\end{theorem}
In fact, their proof immediately gives that the constants $C_1$ and
$C_2$ depend only on a lower bound on the injectivity radius
$\inj( M_0)$, which is half the length of the shortest closed
geodesic.  Our main result is the following refinement of
Theorem~\ref{prop:BSV}:
\begin{restatable}{theorem}{thmmain}\label{thm:main}
    For all closed orientable hyperbolic \3-manifolds $M$ one has
    \begin{equation}\label{eq:main}
      \frac{\pi}{\sqrt{\vol(M)}} \thnorm{\cdotspaced}
      \leq \harnorm{\cdotspaced}
      \leq \frac{10 \pi}{\sqrt{\inj(M)}} \thnorm{\cdotspaced}  
     \mtext{on $H^1(M ; \R)$.}
   \end{equation}
\end{restatable}

We also give families of examples which show that some (but not all)
qualitative aspects of Theorem~\ref{thm:main} are sharp.  The first
proves that the basic form of the first inequality in (\ref{eq:main}) cannot
be improved:
\begin{restatable}{theorem}{thmlowerexs}\label{thm:lowerexs}
  There exists a sequence of $M_n$ and $\phi_n \in H^1(M_n; \R)$
  so that 
  \begin{enumerate}
  \item The quantities $\vol(M_n)$ and $\inj(M_n)$ both go to infinity
    as $n$ does.  
  \item The ratio $\frac{\thnorm{\phi_n}}{\harnorm{\phi_n}
      \sqrt{\vol(M_n)}}$ is constant.  
  \end{enumerate}
\end{restatable}

The next result concerns the second inequality of (\ref{eq:main}), and
shows that the harmonic norm can blow up relative to the Thurston norm 
when the injectivity radius gets small.
\begin{restatable}{theorem}{thminjsmallexs}\label{thm:injsmallexs}
   There exists a sequence of $M_n$ and
   $\phi_n \in H^1(M_n; \R)$ so that
   \begin{enumerate}
     \item The volumes of the $M_n$ are uniformly bounded and
       $\inj(M_n) \to 0$ as $n \to \infty$.  
     \item 
       $\harnorm{\phi_n} \big/ \thnorm{\phi_n} \to \infty$ like
       $\sqrt{-\log\big(\inj(M_n)\big)}$ as $n \to \infty$. 
     \end{enumerate}
\end{restatable}
\noindent
The growth of $\harnorm{\phi_n} \big/ \thnorm{\phi_n}$ in
Theorem~\ref{thm:injsmallexs} is much slower than the most extreme
behavior permitted by the second inequality in (\ref{eq:main}).  We
suspect that the examples in Theorem~\ref{thm:injsmallexs} have
the worst possible behavior, but we are unable to improve (\ref{eq:main})
in that direction and believe doing so requires an entirely new
approach.

%\todo{Either prove this, turn into a conjecture, or delete.}
%Finally, also related to the second inequality in (\ref{eq:main}) but
%now in the context of injectivity radius bounded below we have:
%\begin{theorem}\label{thm:injdefexs}
%There exist a sequence $M_i$ and $\phi_i \in H^1(M_i; \R)$ where
%$\vol(M_i) \to \infty$, $\inj(M_i)$ is constant, $\thnorm{\phi_i} = 2$,
%and $\harnorm{\phi_i} \to A_0 > 0$.
%\end{theorem}

\subsection{Manifolds with large norms}

A very intriguing conjecture of \cite{BergeronSengunVenkatesh2016} is
that for congruence covers of a fixed \emph{arithmetic} hyperbolic
\3-manifold $M_0$, the size of the Thurston norm grows more slowly
than one would naively expect.  Specifically, any cover $M$ of $M_0$
should have a nontrivial $\phi \in H^1(M; \Z)$ where $\thnorm{\phi}$
is bounded by a polynomial in $\vol(M)$; in contrast, the usual
estimates using a natural triangulation of $M$ give only that there is
a $\phi$ with $\thnorm{\phi}$ is at most exponential in
$\vol(M)$. More generally, such estimates give that for any
$\epsilon > 0$ there is a contant $C$ so that any closed hyperbolic
\3-manifold $M$ with $\inj(M) > \epsilon$ and $H^1(M; \Z) \neq 0$ has
a nontrivial $\phi \in H^1(M; \Z)$ where
$\thnorm{\phi} < C^{\vol(M)}$.  Our other contribution here is to show
that such a priori estimates on the size of the Thurston norm cannot
be substantially improved.  Specifically, we give examples of closed
hyperbolic \3-manifolds where the Thurston norm is exponentially large
and yet there is a uniform lower bound on the injectivity radius.
\begin{restatable}{theorem}{thmlargenorm}\label{thm:largenorm}
  There exist constants $C_1, C_2, \epsilon_1 > 0$ and a sequence of
  closed hyperbolic \3-manifolds $M_n$ with $\vol(M_n) \to \infty$ where for
  all $n$:
  \begin{enumerate}
    \item $\inj(M_n) > \epsilon_1$, 
    \item $b_1(M_n) = 1$,
    \item $\thnorm{\phi_n} > C_1 e^{C_2 \vol(M_n)}$ where $\phi_n$ is a
        generator of $H^1(M_n; \Z)$.  \label{item:expgrowth}
  \end{enumerate}
\end{restatable}
\noindent
Presumably, though we do not show this, the examples in
Theorem~\ref{thm:largenorm} are nonarithmetic, and so point toward a
divergence in behavior of the regulator term in the analytic torsion
in the arithmetic and nonarithmetic cases, which is consistent with
e.g.~the experiments of \cite[\S 4]{BrockDunfield2015}.  

\subsection{Proof highlights} The arguments for the two inequalities
of (\ref{eq:main}) in Theorem~\ref{thm:main} are mostly independent.
For the first inequality, we take a quite different approach from that
of \cite{BergeronSengunVenkatesh2016}.  Namely, to mediate between the
topological Thurston norm and the analytical harmonic norm, we study
what we call the \emph{least-area norm}.  For an integral class
$\phi \in H^1(M; \Z)$, this norm is simply the least area of any
embedded smooth surface dual to $\phi$.  Using results on minimal
surfaces of Schoen and Uhlenbeck, we show this new norm is uniformly
comparable to the Thurston norm in the context of hyperbolic
\3-manifolds (Theorem~\ref{thm:areavsthurston}). Also, an argument
using the coarea formula in geometric measure theory allows us to
reinterpret the least-area norm as the $L^1$-analog of the $L^2$-based
harmonic norm (Lemma~\ref{lem:areavsone}).  With these connections in
place, the first inequality of (\ref{eq:main}) boils down to simply
the Cauchy-Schwarz inequality.

In contrast, our proof of the second inequality in (\ref{eq:main})
follows the approach of \cite{BergeronSengunVenkatesh2016} closely.
The key improvement is a refined upper bound on the
$L^\infty$-norm of a harmonic \1-form in terms of its harmonic norm.
The latter result is Theorem~\ref{thm:nosobolev} and is proved by a
detailed analysis of the natural series expansion of a harmonic
\1-form about a point in $\H^3$.

For the examples, Theorem~\ref{thm:lowerexs} uses a simple
construction involving finite covers, and
Theorem~\ref{thm:injsmallexs} is based on Dehn filling a suitable
2-cusped hyperbolic \3-manifold.  Finally, the proof of
Theorem~\ref{thm:largenorm} involves gluing together two acylindrical
homology handlebodies by a large power of a pseudo-Anosov; as with our
prior work on integer homology spheres with large injectivity radius
\cite{BrockDunfield2015}, the key is controlling the homology of the
resulting manifolds.

\subsection{Acknowledgements} Brock was partially supported by US NSF
grant DMS-1207572 and the GEAR Network (US NSF grant DMS-1107452), and
this work was partially conducted at ICERM.  Dunfield was partially
supported by US NSF grant DMS-1106476, the GEAR Network, a Simons
Fellowship, and this work was partially completed while visiting
ICERM, the University of Melbourne, and the Institute for Advanced
Study.  We gratefully thank Nicolas Bergeron, Haluk \Sengun, Akshay
Venkatesh, Anil Hirani, Mike Freedman, Tom Church, Tom Mrowka, and
Peter Kronheimer for helpful conversations and comments.  Finally, we
thank the several referees for their helpful comments and suggestions.

\section{The harmonic and Thurston norms}
\label{sec:norms}

\subsection{Conventions} Throughout this paper, all manifolds of any
dimension are orientable and moreover oriented. All cohomology will
have $\R$ coefficients unless noted otherwise.

\subsection{The harmonic norm}
For a closed Riemannian \3-manifold $M$, the natural positive\hyp definite
inner product on the space $\Omega^k(M)$ of real-valued $k$-forms is
given by
\[
\pair{\alpha, \beta} = \int_M \alpha \wedge *\beta  \mtext{where $*$
  is the Hodge star operator.}
\]
The \emph{harmonic representative} of a class $[\alpha] \in H^k(M)$ is
the unique one that minimizes $\pair{\alpha, \alpha}$; equivalently,
if $\Delta = d d^* + d^* d$ is the Hodge Laplacian, it is the
representative where $\Delta \alpha = 0$.  Thus the cohomology
$H^k(M)$ inherits the above inner product via the identification of it
with the subspace of harmonic forms.  The corresponding norm on
$H^k(M)$ is, by definition, the \emph{harmonic norm}
$\harnorm{\cdotspaced}$ discussed in the introduction.  Equivalently,
it is defined by
\begin{equation}\label{eq:harnormalt}
  \harnorm{\phi} = \inf \setdef{ \harnorm{\alpha}}{\mbox{$\alpha \in
    \Omega^k(M)$ represents $\phi$}}
\end{equation}

For a 1-form $\alpha$ on $M$, a useful geometric viewpoint on
$\harnorm{\alpha}$ is the following. For a point $p$ in $M$, 
denote the operator norm of the linear functional
$\alpha_p \maps T_p M \to \R$ by $\abs{\alpha_p}$; equivalently
$\abs{\alpha_p} = \sqrt{*\big(\alpha_p \wedge * \alpha_p\big)}$ which
is also just the length of $\alpha_p$ under the metric\hyp induced isomorphism
of $T^p M \to T_p M$.  The harmonic norm of $\alpha$ is then simply the
$L^2$-norm of the associated function
$\abs{\alpha} \maps M \to \R_{\geq 0}$ since
\begin{equation}\label{eq:L2asOp}
\harnorm{\alpha} = \sqrt{\int_M \alpha \wedge *\alpha} = \sqrt{\int_M
\abs{\alpha}^2 \ \dVol}
\end{equation}
Analogously, we define the $L^1$- and $L^\infty$-norms of the
1-form $\alpha$ as 
\begin{equation}\label{eq:formnorms}
\onenorm{\alpha} = \int_M \abs{\alpha} \ \dVol \mtext{and}
\infnorm{\alpha} = \max_{p \in M} \abs{\alpha_p}.
\end{equation}

\subsection{The Thurston norm}
For a connected surface, define
$\eulerminus{S} = \max\big( -\chi(S), 0 \big)$; extend this to all
surfaces via
$\eulerminus{S \sqcup S'} = \eulerminus{S} + \eulerminus{S'}$.  For a
compact irreducible \3-manifold $M$, the Thurston norm of
$\phi \in H^1(M; \Z) \cong H_2(M, \partial M; \Z)$ is defined by
\[
\thnorm{ \phi } = \min \setdef{ \eulerminus{S} }{\mbox{$S$ is a properly
    embedded surface dual to $\phi$}}
\]
The Thurston norm extends by continuity to all of $H^1(M)$, and the
resulting unit ball is a finite-sided rational polytope
\cite{Thurston1986}.  When $M$ is hyperbolic, the Thurston norm is
non-degenerate with $\thnorm{\phi} > 0$ for all nonzero $\phi$; in
general, it is only a seminorm.

\section{The least area norm}

To mediate between the harmonic and Thurston norms, we introduce two
additional norms on the first cohomology of a closed Riemannian
\3-manifold $M$, namely the least-area norm and the $L^1$-norm. In
fact, these two norms coincide, but the differing perspectives they
offer are a key tool used to prove the lower bound in
Theorem~\ref{thm:main}.

For $\phi$ in $H^1(M; \Z)$, let $\cF_\phi$ be the
collection of smooth maps $f \maps S \to M$ where $S$ is a closed
oriented surface with $f_*([S])$ dual to $\phi$.  The \emph{least area
  norm} of $\phi$ is
\[
\areanorm{\phi} = \inf \setdef{\Area\left(f(S)\right)}{f \in \cF_\phi}
\]
By standard results in geometric measure theory, the value
$\areanorm{\phi}$ is always realized by a smooth \emph{embedded}
surface $S \subset M$, whose components may be weighted by integer
multiplicities; see e.g.~\cite[Lemma~2.1]{Hass1988} for details.  We
will show below that $\areanorm{\cdotspaced}$ is a seminorm
on $H^1(M; \Z)$ which extends continuously to a seminorm on
all of $H^1(M)$.

In analogy with (\ref{eq:harnormalt}), we use the $L^1$-norm on
1-forms given in (\ref{eq:formnorms}) to define the following function
on $H^1(M)$:
\[
\onenorm{\phi} = \inf \setdef{ \onenorm{\alpha}}{\mbox{$\alpha \in
    \Omega^1(M)$ represents $\phi$}}
\]
Unlike the harmonic norm, the value $\onenorm{\phi}$ is typically not
realized by any smooth form $\alpha$.  Despite this, it is easy to
show that $\onenorm{\cdotspaced}$ is a seminorm on $H^1(M)$.  As
promised, these two new norms are in fact the same:

\begin{lemma}\label{lem:areavsone}
  $\areanorm{\phi} = \onenorm{\phi}$ for all $\phi \in H^1(M; \Z)$.
\end{lemma}
Note that one consequence of Lemma~\ref{lem:areavsone} is the promised
fact that $\areanorm{\cdotspaced}$ extends continuously from
$H^1(M; \Z)$ to a seminorm on all of $H^1(M)$.

\begin{proof}
  To show $\areanorm{\phi} \geq \onenorm{\phi}$, let $S$ be a smooth
  embedded surface dual to $\phi$ of area $\areanorm{\phi}$.  For each
  $\epsilon > 0$, consider a dual 1-form $\alpha_{\epsilon}$ which is
  supported in an $\epsilon$\hyp neighborhood of $S$ and is a slight
  smoothing of
  $\frac{1}{2 \epsilon} d \big(\mbox{signed distance to S}\big)$
  there.  An easy calculation in Fermi coordinates shows that
  $\onenorm{\alpha_{\epsilon}} \to \Area(S) =\areanorm{\phi}$ as
  $\epsilon \to 0$, giving $\areanorm{\phi} \geq \onenorm{\phi}$.

  To establish $\areanorm{\phi} \leq \onenorm{\phi}$, let $\alpha$ be
  any representative of $\phi$.  Since $\phi$ is an integral class, by
  integrating $\alpha$ we get a smooth map $f \maps M \to S^1$ so that
  $\alpha = f^*(\dt)$, where we have parameterized $S^1 = \R/\Z$ by
  $t \in [0, 1]$.  For almost all $t \in [0, 1]$, the set
  $S_t = f^{-1}(t)$ is a smooth surface.  For all $t$, we define
  $\Area(S_t)$ to be the 2-dimensional Hausdorff measure of $S_t$. As
  the operator norm $\abs{\alpha}$ is equal to the 1-Jacobian of the
  map $f$, the Coarea Formula \cite[Theorem 3.8]{Morgan2009} is
  precisely
  \begin{equation}\label{eq:areas}
    \int_M \abs{\alpha} \ \dVol = \int_0^1 \Area(S_t) \ \dt 
  \end{equation}
  Since a function on $[0, 1]$ is less than or equal to its mean on a set of
  positive measure, there are many $t$ so that $S_t$ is smooth
  and $\Area(S_t) \leq \onenorm{\alpha}$.  Taking the infimum over
  representatives $\alpha$ of $\phi$ gives
  $\areanorm{\phi} \leq \onenorm{\phi}$ as desired.
\end{proof}

\subsection{Relationship with the Thurston norm}

When $M$ is hyperbolic, the least area norm is very closely related to
the Thurston norm:
\begin{theorem}\label{thm:areavsthurston}
  For any closed hyperbolic \3-manifold $M$ and  $\phi \in
  H^1(M )$ one has:
  \begin{equation}\label{eq:thurvsone}
    \pi \thnorm{\phi} \leq \areanorm{\phi} \leq 2 \pi
    \thnorm{\phi} 
  \end{equation}
\end{theorem}
The moral behind this result is that \emph{stable} minimal
surfaces in hyperbolic \3-manifolds have uniformly bounded intrinsic
curvature \cite{Schoen1983}, and hence area and genus are essentially
proportionate. Specifically, we will use the following fact, which was
first observed by Uhlenbeck \cite{Uhlenbeck1980} in unpublished work.
\begin{lemma}\label{lem:stable}
  For any stable closed minimal surface $S$ in a hyperbolic
  \3-manifold:
  \begin{equation}\label{eq:stable}
      \pi \eulerminus{S} \leq \Area(S) \leq 2 \pi \eulerminus{S}.
    \end{equation}
\end{lemma}

\begin{proof}[Proof of Lemma~\ref{lem:stable}]
  The proof is essentially the same as \cite[Lemma~6]{Hass1995}, which
  you should see for details.  As $S$ is minimal, its intrinsic
  curvature $K \maps S \to \R$ is bounded above by that of hyperbolic
  space, i.e.~by $-1$.  In particular, by Gauss-Bonnet, every
  component of $S$ has negative Euler characteristic and moreover
  \[
  2 \pi \eulerminus{S} = - 2 \pi \chi (S) = \int_S -K \ \dA \geq
  \int_S 1 \ \dA = 
  \Area(S) 
  \] 
  giving the righthand inequality in (\ref{eq:stable}).  
  For the other inequality, since $S$ is stable, the main argument in
  \cite[Lemma~6]{Hass1995} with the test function $f=1$ gives that 
  $\pi \eulerminus{S} \leq \Area(S)$ as desired.
\end{proof}

\begin{proof}[Proof of Theorem~\ref{thm:areavsthurston}]
  Pick a surface $S$ dual to $\phi$ which is incompressible and
  realizes the Thurston norm, i.e.~$\thnorm{\phi} = \eulerminus{S}$.
  Since $S$ is incompressible, by \cite{FreedmanHassScott1983} we can
  assume that $S$ has least area in its isotopy class and hence is a
  stable minimal surface. Thus by Lemma~\ref{lem:stable} we have
  \[
  \areanorm{\phi} \leq \Area(S) \leq 2 \pi \eulerminus{S} = 2 \pi \thnorm{\phi}.
  \]
  proving the second half of (\ref{eq:thurvsone}). 

  For the other inequality, suppose $S$ is a least area surface dual to $\phi$.
  Note again that $S$ must be a stable minimal surface, and so
  Lemma~\ref{lem:stable} gives
  \[
  \pi \thnorm{\phi} \leq \pi \eulerminus{S} \leq \Area(S) = \areanorm{\phi}
  \]
  proving the rest of (\ref{eq:thurvsone}).
\end{proof}

\section{Pointwise bounds on harmonic 1-forms}

Let $M$ be a closed hyperbolic \3-manifold. A key tool used by
\cite{BergeronSengunVenkatesh2016} in their proof of both inequalities
(\ref{eq:norms}) in Theorem~\ref{prop:BSV} is that there is a constant
$C$, depending somehow on the injectivity radius of the hyperbolic
\3-manifold, so that
\[
\infnorm{\cdotspaced} \leq C \harnorm{\cdotspaced} 
\]
In our parallel Theorem~\ref{thm:main}, we will use this fact only in
the second inequality in (\ref{eq:main}), after first refining it
into:
\begin{theorem}\label{thm:nosobolev}
  If $\alpha$ is a harmonic 1-form on a closed hyperbolic
  manifold $M$ then 
  \begin{equation}\label{eq:supnorm}
  \infnorm{\alpha} \leq \frac{5}{\sqrt{\inj(M)}} \harnorm{\alpha}
  \end{equation}
\end{theorem}
\begin{remark}
  While the $5$ in (\ref{eq:supnorm}) can be improved, it seems
  unlikely that the exponent on $\inj(M)$ can be significantly
  reduced. If $T_{\epsilon}$ is a tube of volume $1$ around a core
  geodesic of length $2 \epsilon$, then there is a harmonic 1-form
  $\alpha_\epsilon$ on $T_\epsilon$ (namely $\dz$ in cylindrical
  coordinates) so that
  \[
  \frac{\infnorm{\alpha_{\epsilon}}}{\harnorm{\alpha_{\epsilon}}} \asymp
  \left( \epsilon \log( \epsilon^{-1} ) \right)^{-1/2}
  \]
\end{remark}

Our proof of Theorem~\ref{thm:nosobolev} starts by understanding how
certain harmonic \1-forms behave on balls in $\H^3$ via
\begin{lemma}\label{lem:dfsharp}
If $f \maps \H^3 \to \R$ is harmonic and $B$ is a ball of 
radius $r$ centered about $p$ then 
\begin{equation}\label{eq:dfbound}
  \abs{\df_p} \leq \frac{1}{\sqrt{\nu(r)}} \harnormB{\df}
\end{equation}
where 
\begin{equation}
  \nu(r) = \nuformula{r} 
\end{equation}
\end{lemma}
The function $\nu \maps \R_{\geq 0} \to \R_{\geq 0}$ is a
monotone increasing bijection with $\nu(r) \sim 4 \pi r^3/3$ as
$r \to 0$ and $\nu(r) \sim 6 \pi r$ as $r \to \infty$.  The estimate
in (\ref{eq:dfbound}) is sharp; in fact, our proof gives a single
harmonic function $f$ for which (\ref{eq:dfbound}) is an equality for
all $r$.  Theorem~\ref{thm:nosobolev} will follow directly from
Lemma~\ref{lem:dfsharp} when $r$ is large, but $\nu(r)$ goes to $0$
too fast as $r \to 0$ to immediately give 
(\ref{eq:supnorm}) when $r$ is small.

The missing ingredient needed to prove Theorem~\ref{thm:nosobolev} is
the following notion.  A \emph{Margulis number}
for $M = \Gamma \big\backslash \H^3$ with $\Gamma \leq \Isom^+(\H^3)$
is a $\mu > 0$ so that for all $p \in \H^3$ the subgroup
\[
\spandef{\gamma \in \Gamma}{ d\big(p, \gamma(p)\big) < \mu}
\]
is abelian.  For example, $\mu = 0.1$ is a Margulis number for any
such $M$ \cite[Theorem~2]{Meyerhoff1987}, and here we will use that
$\mu = 0.29$ is a Margulis number whenever $H^1(M) \neq 0$ by
\cite{CullerShalen2012}.  For any fixed Margulis number $\mu$, define
\[
\thickM = \setdef{m \in M}{\inj_m M \geq \mu/2}  \mtext{and} 
\thinM = \setdef{m \in M}{\inj_m M < \mu/2}
\]
When $M$ is closed, the thin part $\thinM$ is a disjoint union of
tubes about the finitely many closed geodesics of length less than
$\mu$.  We will need:
\begin{lemma}\label{lem:degree}
  Suppose $M$ is a hyperbolic \3-manifold with Margulis number $\mu$,
  and set $\epsilon = \min\left\{ \inj(M), \, \mu/2 \right\}$.
  Let $\pi \maps \H^3 \to M$ be the universal covering map.  If $B
  \subset \H^3$ is a ball of radius $\mu/2$, then 
  \begin{equation}\label{eq:degree}
  \max_{m \in M} \abs{B \cap \pi^{-1}(m)} \leq \frac{\mu}{\epsilon}
  \end{equation}
\end{lemma}

\noindent
We first assemble the pieces and prove Theorem~\ref{thm:nosobolev}
assuming Lemmas~\ref{lem:dfsharp} and \ref{lem:degree}.

\begin{proof}[Proof of Theorem~\ref{thm:nosobolev}]
  We will assume that $H^1(M) \neq 0$ as otherwise the only
  harmonic 1-form is identically zero.  Since $H^1(M) \neq 0$,
  Theorem~1.1 of \cite{CullerShalen2012} gives that $0.29$ is a
  Margulis number for $M$.  Setting $\mu = 0.29$ and
  $\epsilon = \inj(M)$, there are two cases depending on how
  $\epsilon$ compares to $\mu/2$.

  First we do the easy case of when $\epsilon \geq \mu/2$.  Let
  $m \in M$ be a point where $\abs{\alpha_m}$ is maximal.  Take
  $\pi \maps \H^3 \to M$ to be the universal covering map and set
  $\alphatil = \pi^*(\alpha)$. Fix a ball $B \subset \H^3$ of radius
  $\epsilon$ centered at a point $p$ in $\pi^{-1}(m)$.  As $\alpha$ is
  harmonic on the compact manifold $M$, it is both closed and
  coclosed; as these are local properties, the same is true for
  $\alphatil$.  As $\H^3$ is contractible and $\alphatil$ is closed,
  we have $\alphatil = \df$ for some $f \maps \H^3 \to \R$.  Moreover
  $f$ is harmonic since
  $\Delta f = (d^*\circ d) f = d^* \alphatil = 0$.  Using
  Lemma~\ref{lem:dfsharp} we get
  \[
     \infnorm{\alpha} = \abs{\alpha_m} = \abs{\alphatil_p}
     \leq \frac{1}{\sqrt{\nu(\epsilon)}} \harnorm{\alphatil|_B} 
     \leq \frac{1}{\sqrt{\nu(\epsilon)}} \harnorm{\alpha}
  \]
  where the last inequality follows as $\pi|_B$ is injective.  The
  inequality (\ref{eq:supnorm}) now follows from the fact that
  $\sqrt{\epsilon/\nu(\epsilon)} < 3.5 < 5$ for
  $\epsilon \geq \mu/2 = 0.145$.

  Now suppose $\epsilon < \mu/2$. We take the same setup as before
  except that $B \subset \H^3$ will now have radius $\mu/2$.  By
  Lemma~\ref{lem:degree} it follows that
  \[
    \harnorm{\alphatil |_B} 
    \leq \sqrt{\frac{\mu}{\epsilon}} \harnorm{\alpha|_{\pi(B)}} 
    \leq \sqrt{\frac{\mu}{\epsilon}} \harnorm{\alpha} 
  \]
  Hence
  \[
  \infnorm{\alpha} = \abs{\alpha_m} = \abs{\alphatil_p} \leq
  \frac{1}{\sqrt{\nu(\mu/2)}} \harnorm{\alphatil|_B} \leq
  \sqrt{\frac{\mu}{\nu(\mu/2)}} \frac{1}{\sqrt{\epsilon}} \harnorm{\alpha}
  \]
  As $\sqrt{\mu/\nu(\mu/2)} \approx 4.78$ at $\mu = 0.29$, we have
  proved (\ref{eq:supnorm}) in this case as well.
\end{proof}

Turning to the lemmas, we start with the easier one which follows from
a simple geometric argument. 

\begin{proof}[Proof of Lemma~\ref{lem:degree}]
  We can assume $\epsilon = \inj(M) < \mu/2$ as otherwise $\pi|_B$ is
  injective and the result is immediate since the righthand side of
  (\ref{eq:degree}) is $2$.  The basic idea of the proof is that the
  worst-case senario is when $m$ is on a closed geodesic $C$ of
  minimal length $2 \epsilon$, and $B$ is centered at a point of
  $\pi^{-1}(C)$. Then $B$ can contain $n + 1$ points in $\pi^{-1}(C)$
  where $n = \lfloor \mu/ 2\epsilon \rfloor$; using that
  $n+1 \leq \mu/\epsilon$ then gives (\ref{eq:degree}).  We now give
  the detailed proof.

  If $m \in \thickM$, then any pair of distinct points in
  $\pi^{-1}(m)$ are distance at least $\mu$ apart, and hence at most
  one is in the open ball $B$; as $\mu/\epsilon \geq 2$, we have
  proven  (\ref{eq:degree}) in this case. 

  If $m \in \thinM$, it lies in some tube $T$ about a short closed
  geodesic $C$.  The components of $\pi^{-1}(T)$ are radius $R$
  neighborhoods about the various geodesic lines in $\pi^{-1}(C)$.
  First, note that $B \cap \pi^{-1}(M)$ must lie in a single component
  $\Ttil$ of $\pi^{-1}(T)$; let $\gamma \in \Gamma$ generate the
  stabilizer of $\Ttil$.  Pick a $\mtil_0 \in \Ttil \cap \pi^{-1}(m)$;
  then $\pi^{-1}(m)$ consists of $\gamma^n \cdot \mtil_0$ for
  $n \in \Z$.  Adjust $\mtil_0$ if necessary so that $\mtil_0 \in B$
  and any $\mtil_n = \gamma^n \cdot \mtil_0$ in $B$ has $n \geq 0$.
  Since
  $d(\mtil_0, \mtil_n) \geq n \cdot \mathrm{len}(C) \geq 2 n
  \epsilon$,
  if $\mtil_n \in B$ we have $n \leq \mu / 2\epsilon$.  So there are
  at most $(\mu / 2\epsilon) + 1$ of the $\mtil_i$ in $B$.  Since
  $2 \epsilon \leq \mu$, we get $\abs{B \cap \pi^{-1}(m)} \leq
  \mu/\epsilon$ as desired. 
\end{proof}

While the calculations in the proof of Lemma~\ref{lem:dfsharp} are
somewhat involved (unsurprisingly given the formula for $\nu(r)$), the
basic idea is simple and we sketch it now.  Using the natural series
expansion for harmonic functions about $p$, we show that after a
suitable isometry of $\H^3$ fixing $p$ we have
\[
\df = a \omega + \beta
\]
where $a \in \R_{\geq 0}$, the 1-form $\omega$ is a fixed and
independent of $f$ with $\abs{\omega_p} = 1$, the 1-form $\beta$
vanishes at $0$, and $\beta$ is orthogonal to $\omega$ in
$L^2\big(B_r(p)\big)$ for all $r$.  Then $\df_p = a \omega_p$ and for
each such $B = B_r(p)$ the orthogonality of $\omega$ and $\beta$
implies $a \harnormB{\omega} \leq \harnormB{\df}$.
Hence 
\[
\abs{\df_p} = \abs{a \omega_p} = 
a \leq \frac{\harnormB{\df}}{\harnormB{\omega}}
\]
Thus we simply define $\nu(r)$ to be $\norm{\omega}_{L^2(B_r(p))}^2$
and Lemma~\ref{lem:dfsharp} will then follow by calculating $\nu(r)$
explicitly.

\subsection{Series expansions of harmonic functions}  

We now describe in detail the key tool used to prove
Lemma~\ref{lem:dfsharp}: that every harmonic function
$f \maps \H^3 \to \R$ has a series expansion in terms of a certain
basis $\{ \psilm \}$ of harmonic functions centered around $p$;
throughout, see \cite{Minemura1973} or \cite[\S
3.5]{ElstrodtGrunewaldMennicke1998} for details.  Consider the
spherical coordinates
$(r, \phi, \theta) \in [0, \infty) \times [0, \pi) \times [0, 2 \pi)$
on $\H^3$ centered about $p$; in these coordinates, the metric is:
\[
 \ds^2_{\H^3} = \dr^2 + \sinh^2(r) \ds^2_{S^2} = \dr^2 + \sinh^2(r)
 \left( \dphi^2 + \sin^2(\phi) \dtheta^2 \right)
\]
For $\ell \in \Z_{\geq 0}$, define $\psil \maps \R_{\geq 0} \to
\R_{\geq 0}$ by 
\[
\psil(r) = \frac{\Gamma\left(\frac{3}{2}\right)\Gamma(\ell +2)}{\Gamma\left(\ell +
  \frac{3}{2}\right)} \tanh^{\ell} \left( \frac{r}{2} \right) \cdot \, 
\twoFone{-\frac{1}{2}}{\ell}{\ell + \frac{3}{2}}{\tanh^2
  \left(\frac{r}{2} \right)}
\]
where ${}_2 F_1$ is the usual hypergeometric function and $\Psi_0$ is
simply the constant function $1$.  If $\Ylm(\phi, \theta)$ for
$\ell \geq 0$ and $-\ell \leq m \leq \ell$ are the usual basis for the
\emph{real} spherical harmonics on $S^2$, define:
\[
\psilm = \psil(r) \Ylm(\phi, \theta)
\]
These are harmonic functions on all of $\H^3$, and moreover every
harmonic function $f \maps \H^3 \to \R$ has unique 
$\alm \in \R$ so that
\begin{equation}\label{eq:series}
  f = \sum_{\ell = 0}^\infty \sum_{m = -\ell}^{\ell} \alm \psilm 
\end{equation}
where the series converges absolutely and uniformly on compact
subsets of $\H^3$.  We will use the following elementary properties of
these functions:

\begin{lemma}\label{lem:psilm}
  \leavevmode
  \begin{enumerate}
  \item \label{item:ordervanish} $\psil(r)$ vanishes to order exactly
    $\ell$ at $r = 0$. Consequently, $\psilm$ vanishes to order
    exactly $\ell - 1$ at $p$.

  \item On any ball $B$ about $p$, the functions $\psilm$ are
    orthogonal in $L^2(B)$.

  \item The 1-forms $\omegalm = d \psilm$ are also orthogonal in
    each $\Omega^1(B)$.  
       
  \end{enumerate}
\end{lemma}
\begin{proof}
  The claim (a) follows since
  $\tanh \frac{r}{2} = \frac{1}{2} r + O(r^2)$ for small $r$ and in
  addition $\twoFone{a}{b}{c}{0} = 1$.  The second claim (b) is an
  easy consequence of the fact that the $\Ylm$ are orthonormal as
  elements of $L^2\left(S^2\right)$.  For part (c), we have
  \[
  \omegalm = d\left(\psil \Ylm\right) = \Ylm \frac{\partial
    \psil}{\partial r} \dr + \psil d\Ylm
  \]
  and then observing that the cross-terms vanish we get
  \begin{equation}\label{eq:omegawedge}
  \omegalm \wedge *\omega_{k n} = \Ylm Y_{k n} \frac{\partial
    \psil}{\partial r} \frac{\partial
    \psi_k}{\partial r} \dVol + \psil \psi_k d\Ylm \wedge * dY_{kn}.
  \end{equation}
  To compute $\pair{\omegalm, \ \omega_{k n}}$, we integrate the above
  over $B$ and show it vanishes unless $(\ell, m) = (k, n)$.  In fact,
  we argue that the integral of (\ref{eq:omegawedge}) over each $S^2$
  where $r$ is fixed is $0$.  For the first term on the right-hand
  side of (\ref{eq:omegawedge}) this follows immediately from the
  orthogonality of the $\Ylm$ in $L^2\left(S^2\right)$.  For the
  second term, note that
  $* dY_{kn} = \left(\starbar dY_{kn} \right)\wedge \dr$ where
  $\starbar$ is the Hodge star operator on $S^2$, and hence the real
  claim is that $\pair{d\Ylm, dY_{kn}}_{S^2}$ vanishes. This can be
  deduced from the orthogonality of $\Ylm$ and $Y_{kn}$ via
  \[
    \pair{d\Ylm, dY_{kn}}_{S^2} =  \pair{\Ylm,  d^{\starbar} d
    Y_{kn}}_{S^2}  = \pair{\Ylm,  \Delta_{S^2} Y_{kn}}_{S^2} 
  = k(k+1) \pair{\Ylm,  Y_{kn}}_{S^2} 
    \]
    where the last equality holds because $Y_{kn}$ is an eigenfunction of
    $\Delta_{S^2}$.  This finishes the proof of (c).  
\end{proof}

We now prove Lemma~\ref{lem:dfsharp} using the approach
sketched earlier.

\begin{proof}[Proof of Lemma~\ref{lem:dfsharp}]
  Let the $\alm$ be the coefficients in the expansion
  (\ref{eq:series}) for $f$, and note we get a corresponding
  expansion
  \[
  \df = \sum_{\ell = 1}^\infty \sum_{m = -\ell}^{\ell} \alm \omegalm
  \mtext{where $\omegalm = d \psilm$.}
  \]
  Here the series converges absolutely and uniformly on $B$, and
  henceforth we view $B$ as the domain of all our 1-forms.  Defining
  \[
  \eta = a_{1,-1} \omega_{1,-1} + a_{1,0} \omega_{1,0} + a_{1,1}
  \omega_{1,1}
  \]
  we see by Lemma~\ref{lem:psilm}(\ref{item:ordervanish}) that
  $\eta_p = \df_p$.  Because of the orthogonality of the $\omegalm$ on
  $B$, we know $\harnorm{\eta} \leq \harnorm{\df}$ and so it suffices
  to prove (\ref{eq:dfbound}) for $\eta$, or indeed for the components of
  $\eta$.  In fact, because we can use an isometry of $\H^3$ fixing
  $0$ to interchange the $Y_{1,m}$, it suffices to establish
  (\ref{eq:dfbound}) for the single form $\omega_{1,0}$, or indeed any
  multiple of it. Thus Lemma~\ref{lem:dfsharp} will follow immediately
  from the next result.  
\end{proof}

\begin{lemma}
  For the harmonic 1-form $\omega = \sqrt{3 \pi} \cdot \omega_{1,0}$
  we have $\abs{\omega_p} = 1$ and $\norm{\omega}_{L^2(B_r(p))} =
  \nu(r)$.  
\end{lemma}

\begin{proof}
  The $\psil$ actually have alternate expressions in terms of
  elementary functions; in the case of interest, using that
  \[
  \twoFone{1/2}{1}{3/2}{x^2} = \sum_{n=0}^\infty \frac{(1/2)_n
    (1)_n}{(3/2)_n} \frac{x^{2n}}{n!} = 
     \sum_{n = 0}^\infty \frac{x^{2n}}{2n+1} =
  \frac{\arctanh(x^2)}{x} 
  \]
  and applying two contiguous relations for hypergeometric functions
  yields
  \[
  \twoFone{-1/2}{1}{5/2}{x^2} = 
      \frac{3 \left(x^3-\left(x^2-1\right)^2 \arctanh(x)+x\right)}{8 x^3}
  \]
  and hence 
  \[
  \psi_1(r) = \coth(r) - r \csch^2(r) = \frac{\sinh(r)\cosh(r) -
    r}{\sinh^2(r)} = \frac{2}{3} r -\frac{4}{45} r^3 + O(r^5).
  \]
  As $Y_{1,0}$ is $\sqrt{3/4\pi}\cos{\phi}$, writing 
  \[
  \omega =   \sqrt{3 \pi} \cdot \omega_{1,0} = \frac{3}{2} d\left( \psi_1
    \cos \phi \right)
  \] 
  in terms of the orthonormal coframe
  \[
  \drhat = \dr \qquad \dphihat = \sinh(r) \dphi \qquad \dthetahat = \sinh(r)
  \sin(\phi) \dtheta 
  \]
  gives
  \[
  \omega = \frac{3}{2} \left( \left(\partial_r \psi_1\right)
    \cos(\phi) \drhat
    - \frac{\psi_1 \sin(\phi)}{\sinh(r)} \dphihat \right)
    \mtext{where $\partial_r \psi_1 = 
    \displaystyle \frac{\partial \psi_1}{\partial r}
    = 2 \cdot \frac{r\coth(r) -1}{\sinh^2(r)}$,}
  \]
  and hence
  \[
  \abs{\omega}^2 = \frac{9}{4} \left( \left(\partial_r \psi_1\right)^2
    \cos^2(\phi)  + \frac{\psi_1^2 \sin^2 \phi}{\sinh^2(r)} \right)
  \]
  Approaching the origin along the ray $\phi=0$ gives:
  \[
  \abs{\omega_p} = \frac{3}{2} \lim_{r \to 0} \frac{\partial
    \psi_1}{\partial r} = 3 \lim_{r \to 0} \frac{r
    \coth(r) -1}{\sinh^2(r)} = 1
  \]
  Computing the $L^2$-norm of $\omega$ on $B$ gives
  \begin{equation}
    \begin{split}
  \harnorm{\omega}^2 &= \int_B \abs{\omega}^2 \dVol =
      \int_0^{R} \int_0^{\pi} \int_0^{2 \pi} \abs{\omega}^2 \sinh^2(r)
      \sin(\phi) \ \dtheta \dphi \dr\\
      &= \frac{9 \pi}{2} \int_0^{R} \int_0^{\pi} 
\left(\partial_r \psi_1\right)^2 \sinh^2(r) \cos^2(\phi) \sin(\phi) + \psi_1^2
      \sin^3(\phi) \ \dphi \dr \\
      &= 3 \pi \int_0^{R} 
      \left(\partial_r \psi_1\right)^2 \sinh^2(r) + 2
      \psi_1^2 \ \dr \\
      &=  6 \pi \int_0^{R} \coth ^2(r)+2\csch^2(r) -6 r \coth (r)
        \csch^2(r) \\
        & \hspace{4cm}+  r^2 \csch^2(r) \left(2 \coth^2(r)+\csch^2(r)\right)\dr  \\
      &=  \nuformula{R}
      \end{split}
  \end{equation}
  which proves the lemma.
\end{proof}

\section{Proof of Theorem~\ref{thm:main}}

This section is devoted to the proof of 

\thmmain*

\begin{proof}[Proof of Theorem~\ref{thm:main}]
  We start with the lower bound in (\ref{eq:main}), where we use the
  two guises of the least-area/$L^1$--norm to mediate between the
  Thurston and harmonic norms and thereby reduce the claim to the
  Cauchy-Schwarz inequality.  Suppose $\phi \in H^1(M)$ and let
  $\alpha$ be the harmonic 1-form representing $\phi$.  By
  Theorem~\ref{thm:areavsthurston} and Lemma~\ref{lem:areavsone} we
  have
  \begin{equation*}
      \pi \thnorm{\phi} \leq \areanorm{\phi} = \onenorm{\phi} 
  \end{equation*}
  From the definition of the $L^1$-norm, we have
  $\onenorm{\phi} \leq \onenorm{\alpha}$, and applying Cauchy-Schwarz
  to the pair $\abs{\alpha} \maps M \to \R$ and the constant function
  1 gives
  \begin{equation}\label{eq:cs}
    \pi \thnorm{\phi} \leq \onenorm{\alpha} 
    =  \onenorm{ \abs{\alpha} \cdot 1 } \leq \harnorm{ \alpha }  \harnorm{
      1 } = \harnorm{ \alpha } \sqrt{\vol(M)}
  \end{equation}
  Since $\harnorm{\phi} = \harnorm{\alpha}$ by definition, dividing
  (\ref{eq:cs}) through by $\sqrt{\vol(M)}$ gives 
  the first part of (\ref{eq:main}).

  The proof of the upper-bound in (\ref{eq:main}) is essentially the
  same as given in \cite{BergeronSengunVenkatesh2016} for the
  corresponding part of (\ref{eq:norms}), but using the upgraded
  Theorem~\ref{thm:nosobolev} to relate $\infnorm{\cdotspaced}$ and
  $\harnorm{\cdotspaced}$ and so give a sharper result.  By continuity
  of the norms, it suffices to prove the upper bound for
  $\phi \in H^1(M; \Z)$.  Using Theorem~\ref{thm:areavsthurston}, fix
  a surface $S$ dual to $\phi$ of area at most $2 \pi \thnorm{\phi}$;
  by definition, this means that for every \emph{closed} 2-form
  $\beta$ on $M$ one has $\int_M \beta \wedge \alpha = \int_S \beta$.
  If $\alpha$ is the harmonic representative of $\phi$, then
  $d^* \alpha = -*d*\alpha = 0$, and so it follows that $*\alpha$ is
  closed.  Hence
  \begin{equation}
    \begin{split}
      \harnorm{\alpha}^2 &= \int_M \alpha \wedge
      *\alpha = \int_M *\alpha \wedge \alpha  = \int_S *\alpha \\
      &\leq \int_S \abs{*\alpha} \ \dA  =  \int_S
      \abs{\alpha} \ \dA \leq \int_S \infnorm{\alpha} \, \dA  \\
      &\leq \infnorm{\alpha} \Area(S) \leq  2 \pi \infnorm{\alpha} \thnorm{\phi}.
    \end{split}
  \end{equation}
  Applying (\ref{eq:supnorm}) from Theorem~\ref{thm:nosobolev}, we get 
  \[
  \harnorm{\alpha}^2 \leq \frac{10 \pi}{\sqrt{\inj(M)}} \harnorm{\alpha}
  \thnorm{\phi}
  \]
  Dividing through by $\harnorm{\alpha}$ gives the upper bound in 
  (\ref{eq:main}), proving the theorem. 
\end{proof}

\subsection{The gauge theory viewpoint}\label{sec:gauge}

As mentioned in the introduction, Kronheimer and Mrowka found a
striking relationship between the Thurston norm and harmonic norms as
one varies the Riemannian metric.  Specifically, if $h$ is a
Riemannian metric on $M$, let $\hnorm{\cdotspaced}$ denote the induced
harmonic norm on $H^1(M; \R)$.  They proved:

\begin{theorem}[{\cite{KronheimerMrowka1997b}}]\label{thm:KM}
  Suppose $M$ is a closed oriented irreducible \3-manifold.  Then
  for all $\phi \in H^1(M;\R)$ one has
  \begin{equation}\label{eq:KM}
    4\pi \thnorm{\phi} = \inf \setdef{\ \hnorm{s_h} \cdot \hnorm{\phi} \  }{\mbox{$h$ is
      a Riemannian metric on $M$}}
  \end{equation}
  where $s_h$ is the scalar curvature of $h$.  
\end{theorem}
Theorem~\ref{thm:KM} above is equivalent to Theorem 2 of
\cite{KronheimerMrowka1997b}, as we explain below.  Kronheimer and
Mrowka pointed out to us that, in this form, specializing
(\ref{eq:KM}) to a hyperbolic metric $h$ on $M$ gives
\[
\frac{2 \pi}{3 \sqrt{\vol(M)}} \thnorm{\cdotspaced}\leq \harnorm{\cdotspaced}
\]
since the scalar curvature of a hyperbolic metric is $-6$; this is
only slightly weaker than the first inequality of
(\ref{eq:main}).

Theorem~\ref{thm:KM} gives another perspective on
Theorem~\ref{thm:main}, namely that (\ref{eq:main}) bounds, from above
and below, how close the hyperbolic metric can be to realizing the
infimum in (\ref{eq:KM}).  Theorem~\ref{thm:KM} is a corollary of
previous deep work of Kronheimer and Mrowka
\cite{KronheimerMrowka1997a} which in turn depends on results of Gabai
\cite{Gabai1983}, Eliashberg-Thurston \cite{EliashbergThurston1998},
and Taubes \cite{Taubes1996}.  In contrast, our proof here of the
first inequality of (\ref{eq:main}) uses only some basic facts about
minimal surfaces.  It would be very interesting to try to extend our
approach to arbitrary metrics and prove results along the lines of
Theorem~\ref{thm:KM}; if successful, this would provide a new
perspective on their results.

\begin{proof}[Proof of Theorem~\ref{thm:KM}] Here is why
  Theorem~\ref{thm:KM} is equivalent to Theorem 2 of
  \cite{KronheimerMrowka1997b}, which is stated in terms of the
  \emph{dual} Thurston norm on $H^2(M; \R)$.
  Recall that for a norm $x$ on a
  finite-dimensional real vector space $V$, the dual norm $x^*$ on
  $V^*$ is defined by
  \[
  x^*(\psi) = \sup \setdef{ \psi(v) }{\mbox{$v \in V$ with $x(v) \leq 1$}}
  \]
  Theorem 2 of
  \cite{KronheimerMrowka1997b} says that for all $\psi \in H^2(M; \R)$
  one has
  \begin{equation}\label{eq:KMorig}
    \dualthnorm{\psi} = 4 \pi \sup_h \frac{\hnorm{\psi_h}}{\hnorm{s_h}}
  \end{equation}
  where $\dualthnorm{\cdotspaced}$ on $H^2(M; \R)$ is the norm dual to
  the usual Thurston norm on $H_2(M; \R)$.   It is not hard to show that the
  harmonic norms on $H^1(M; \R)$ and $H^2(M; \R)$ are dual for any
  fixed $h$.  Hence if $c_h = \hnorm{s_h}/4 \pi$ the norm
  $c_h \hnorm{\cdotspaced}$ on $H^1(M; \R)$ is dual to
  $\frac{1}{c_h} \hnorm{\cdotspaced}$ on $H^2(M; \R)$.  The
  equivalence of (\ref{eq:KM}) and (\ref{eq:KMorig}) now follows from
  the following simple fact: if $x_n$ are norms on $V$ where
  $x = \sup x_n$ is finite on $V$, then $x$ is a norm with
  $x^* = \inf x_n^*$.
\end{proof}

\section{Families of examples}

This section is devoted to proving Theorems~\ref{thm:lowerexs} and
\ref{thm:injsmallexs}, starting with the former as it is easier. 

\thmlowerexs*

\begin{proof}[Proof of Theorem~\ref{thm:lowerexs}]
  The examples $M_n$ will be built from a tower of finite covers,
  where the $\phi_n$ are pullbacks of some fixed class
  $\phi_0 \in H^1(M_0)$.  To analyze that situation, first consider a
  degree $d$ covering map $\pi: \Mtil \to M$ and some
  $\phi\in H^1(M)$; as we now explain, the norms of $\phi$ and
  $\pi^*(\phi)$ differ by factors depending only on $d$.  If $\alpha$
  is the harmonic representative of $\phi$, then as being in the
  kernel of the Laplacian is a local property, the form
  $\pi^*(\alpha)$ must be the harmonic representative of
  $\pi^*(\phi)$.  It follows that
  \[
  \harnorm{\pi^*(\phi)} = \sqrt{d} \cdot \harnorm{\phi}.
  \]
  In contrast, it is a deep theorem of Gabai
  \cite[Cor.~6.13]{Gabai1983} that 
  \[
  \thnorm{\pi^*(\phi)} = d \cdot \thnorm{\phi}
  \]
  Thus, the ratio
  \[
  \frac{\thnorm{\cdotspaced}}{\harnorm{\cdotspaced} \sqrt{\vol{(\cdotspaced)}}}
  \]
  is the same for both $(M, \phi)$ and $\left(\Mtil, \pi^*(\phi)\right)$.  

  To prove the theorem, let $M_0$ be any closed hyperbolic 3-manifold
  with a nonzero class $\phi_0\in H^1(M_0)$ and choose a tower of
  finite covers 
  \[
  M_0 \leftarrow M_1 \leftarrow M_2 \leftarrow M_3 \leftarrow \cdots
  \]
  where $\inj(M_n) \to \infty$; this can be done since $\pi_1(M_0)$ is
  residually finite, indeed residually simple, see e.g.~\cite{LongReid1998}.

  Taking $\phi_n \in H^1(M_n)$ to be the pullback of $\phi$, we have
  constructed pairs $(M_n, \phi_n)$ which have all the claimed
  properties.
\end{proof}

\subsection{Harmonic forms on tubes}  To prove
Theorem~\ref{thm:injsmallexs}, we will need:

\begin{lemma}\label{lem:tube}
  Suppose $V$ is a tube in a closed hyperbolic \3-manifold $M$ with a
  core $C$ of length $\epsilon$ and depth $R$.  If $\alpha$ is a
  harmonic 1-form on $M$ with $\int_C \alpha = 1$ then
  \begin{equation}
    \harnorm{\alpha|_V} \geq
      \sqrt{\frac{2 \pi}{\epsilon} \log\left(\cosh R \right)}
  \end{equation}
\end{lemma}
Before proving the lemma, we give coordinates on the tube $V$ and
do some preliminary calculations.  Specifically, consider cylindrical
coordinates
$(r, \theta, z) \in [0, R] \times [0, 2\pi] \times [0, \epsilon]$ with
the additional identification
$(r, \theta, \epsilon) \sim (r, \theta + \theta_0, 0)$, where the
twist angle $\theta_0$ is determined by the geometry of the tube; the
metric on $V$ is then given by
\[ 
g_{\H^3} = \dr^2 + \sinh^2(r) \dtheta^2 + \cosh^2(r) \dz^2
\]
Since $\left\{ \dr, \sinh(r) \dtheta, \cosh(r) \dz \right\}$ gives an
orthonormal basis of 1-forms at each point of $V$, we have
\begin{equation}\label{eq:voltube}
\Vol(V) = \int_V \dVol = 
  \int_0^\epsilon \int_0^{2 \pi} \int_0^R \sinh(r) \cosh(r) \dr
  \dtheta \dz
  = \pi \epsilon \sinh^2(R)
\end{equation}
Note that the form $\dz$ is compatible with the identification of $z =
0$ with $z = \epsilon$, giving a 1-form on $V$.  Now the form
$\dz$ is closed and also coclosed since
\[
d^*(\dz) = -*d*\dz = - * d \big(\tanh(r) \dr \wedge \dtheta\big) 
  = - * 0 = 0
\]
and hence $\dz$ is harmonic.  Notice $\omega = \frac{1}{\epsilon} \dz$ is
a reasonable candidate for $\alpha|_V$ in Lemma~\ref{lem:tube} as
$\int_C \omega = 1$.  For this form, we have
\begin{equation}
\begin{split}\label{eq:normomega}
\harnorm{\omega}^2 &= \int_V \omega \wedge * \omega 
= \frac{1}{\epsilon^2} \int_V \dz \wedge \big(\tanh(r) \dr \wedge
  \dtheta\big) \\ 
&= \frac{1}{\epsilon^2} \int_0^\epsilon \int_0^{2\pi} \int_0^R \tanh(r) \dr 
  \dtheta \dz
 = \frac{2 \pi}{\epsilon} \log(\cosh R)
\end{split}
\end{equation}
Thus Lemma~\ref{lem:tube} can be interpreted as saying that the
harmonic norm of $\alpha|_V$ is at least that of this explicit
$\omega$, and we take this viewpoint in the proof itself.

\begin{proof}[Proof of Lemma~\ref{lem:tube}]
  From now on, we denote $\alpha|_V$ by $\alpha$ and we use only that
  $\alpha$ is closed, coclosed, and $\int_C \alpha = 1$.  There is an
  action of $G = S^1 \times S^1$ on $V$ by isometries, namely
  translations in the $\theta$ and $z$ coordinates.  First, we show it
  suffices to prove the lemma for the average of $\alpha$ under this
  action, namely
  \[
  \alphabar = \int_{G} g^*(\alpha) \dg \mtext{where $\dg$ is Haar
    measure on $G$.}
  \]
  The advantage of $\alphabar$ will be that it must be
  $G$-invariant.  Note that $\alphabar$ can be $C^\infty$-approximated
  by finite averages over suitably chosen finite subsets $\{g_i\}$ of
  $G$:
  \begin{equation}\label{eq:finiteave}
  \alphabar \approx \frac{1}{N} \sum^N_i g_i^*(\alpha)
  \end{equation}
  and from this it follows that $\alphabar$ is closed, coclosed, and
  $\int_C \alphabar = 1$.  Moreover, since all
  $\harnorm{g_i^*(\alpha)} = \harnorm{\alpha}$, the triangle
  inequality applied to (\ref{eq:finiteave}) gives that
  $\harnorm{\alphabar} \leq \harnorm{\alpha}$.  This shows we need
  only consider the $G$-invariant form $\alphabar$.

  We next show that $\alphabar$ is equal to 
  $\omega = \frac{1}{\epsilon} \dz$; combined with the
  calculation of $\harnorm{\omega}$ in (\ref{eq:normomega}), this will
  establish the lemma.  Since the \1-forms
  $\{ \dr, \dtheta, \dz \}$ are $G$-invariant, it follows that
  $\alphabar$ can be expressed as
  \[
  \alphabar = a(r) \dr + b(r) \dtheta  + c(r) \dz
  \]
  Using that $\dalphabar = 0$, we get that $b$ and $c$ must be
  constants; as $\abs{\dtheta} = 1/\sinh(r)$, we must further have
  $b=0$ so that $\alphabar$ makes sense along the core $C$ where
  $r=0$.  As $\alphabar$ is coclosed, we learn that
  \[
  \begin{split}
    0 &= d(*\alphabar) = d\big( a(r) \sinh(r) \cosh(r) \dtheta \wedge \dz +
      c \tanh(r) \dr \wedge \dtheta \big) \\
      & = \frac{\partial}{\partial r} \big( a(r) \sinh(r) \cosh(r)
      \big) \dr \wedge \dtheta \wedge \dz
  \end{split}
  \]
  Consequently, $a(r) \sinh(r) \cosh(r)$ is constant. Moreover, that
  constant must be $0$ to prevent $a(r)$ from blowing up as $r \to
  0$.  So we know $\alphabar = c \dz$ and finally that $c$ must be
  $1/\epsilon$ to ensure $\int_C \alphabar = 1$.  So $\alphabar$ is
  equal to $\omega$ as claimed, proving the lemma. 
\end{proof}

\thminjsmallexs*

\begin{figure}[tb]
  \begin{center}
  \begin{tikzpicture}[scale=0.7, rotate=-40, line width=1.5, line cap=round, line join=round]
  \begin{scope}
    \draw (4.72, 4.57) .. controls (4.27, 4.57) and (3.81, 4.73) .. 
          (3.81, 5.12) .. controls (3.81, 5.58) and (4.39, 5.68) .. (4.92, 5.68);
    \draw (4.92, 5.68) .. controls (5.23, 5.68) and (5.55, 5.68) .. (5.86, 5.68);
    \draw (6.27, 5.68) .. controls (6.56, 5.68) and (6.86, 5.68) .. (7.15, 5.68);
    \draw (7.15, 5.68) .. controls (7.69, 5.68) and (8.27, 5.58) .. 
          (8.27, 5.12) .. controls (8.27, 4.73) and (7.81, 4.57) .. (7.36, 4.57);
    \draw (6.95, 4.57) .. controls (6.65, 4.57) and (6.35, 4.57) .. (6.05, 4.57);
    \draw (6.05, 4.57) .. controls (5.75, 4.57) and (5.44, 4.57) .. (5.13, 4.57);
  \end{scope}
  \begin{scope}
    \draw (6.05, 3.25) .. controls (6.04, 2.80) and (5.88, 2.34) .. 
          (5.48, 2.34) .. controls (5.02, 2.34) and (4.92, 2.92) .. (4.92, 3.45);
    \draw (4.92, 3.45) .. controls (4.92, 3.82) and (4.92, 4.19) .. (4.92, 4.57);
    \draw (4.92, 4.57) .. controls (4.92, 4.87) and (4.92, 5.17) .. (4.92, 5.48);
    \draw (4.92, 5.89) .. controls (4.92, 6.43) and (4.39, 6.80) .. 
          (3.81, 6.80) .. controls (2.69, 6.80) and (2.69, 5.10) .. (2.69, 3.66);
    \draw (2.69, 3.25) .. controls (2.69, 2.57) and (2.69, 1.90) .. (2.69, 1.22);
    \draw (2.69, 1.22) .. controls (2.69, 0.61) and (2.19, 0.11) .. 
          (1.58, 0.11) .. controls (0.46, 0.11) and (0.76, 2.62) .. 
          (1.01, 4.66) .. controls (1.31, 7.17) and (3.40, 9.10) .. (5.88, 8.95);
    \draw (6.29, 8.92) .. controls (7.27, 8.86) and (8.26, 8.80) .. (9.24, 8.74);
    \draw (9.24, 8.74) .. controls (9.84, 8.70) and (10.38, 8.31) .. 
          (10.43, 7.73) .. controls (10.47, 7.22) and (10.03, 6.80) .. (9.50, 6.80);
    \draw (9.09, 6.80) .. controls (8.17, 6.80) and (7.15, 6.67) .. (7.15, 5.89);
    \draw (7.15, 5.48) .. controls (7.15, 5.17) and (7.15, 4.87) .. (7.15, 4.57);
    \draw (7.15, 4.57) .. controls (7.15, 4.19) and (7.15, 3.82) .. (7.15, 3.45);
    \draw (7.15, 3.45) .. controls (7.15, 1.69) and (4.95, 1.22) .. (2.90, 1.22);
    \draw (2.49, 1.22) .. controls (1.94, 1.22) and (1.58, 1.75) .. 
          (1.58, 2.34) .. controls (1.58, 2.95) and (2.08, 3.45) .. (2.69, 3.45);
    \draw (2.69, 3.45) .. controls (3.37, 3.45) and (4.04, 3.45) .. (4.72, 3.45);
    \draw (5.13, 3.45) .. controls (5.44, 3.45) and (5.74, 3.45) .. (6.05, 3.45);
    \draw (6.05, 3.45) .. controls (6.35, 3.45) and (6.65, 3.45) .. (6.95, 3.45);
    \draw (7.36, 3.45) .. controls (8.78, 3.45) and (9.34, 5.15) .. (9.29, 6.80);
    \draw (9.29, 6.80) .. controls (9.28, 7.37) and (9.26, 7.95) .. (9.25, 8.53);
    \draw (9.24, 8.94) .. controls (9.22, 9.60) and (8.40, 9.78) .. 
          (7.65, 9.82) .. controls (6.88, 9.86) and (6.09, 9.60) .. (6.08, 8.93);
    \draw (6.08, 8.93) .. controls (6.08, 7.85) and (6.07, 6.76) .. (6.06, 5.68);
    \draw (6.06, 5.68) .. controls (6.06, 5.38) and (6.06, 5.07) .. (6.06, 4.77);
    \draw (6.05, 4.36) .. controls (6.05, 4.13) and (6.05, 3.89) .. (6.05, 3.66);
  \end{scope}
  \end{tikzpicture}
  \end{center}

  \vspace{-2.5ex}

  \caption{The exterior $W$ of the link $L = L14n21792$ above is one
    example that satisfies the conditions used in the proof of
    Theorem~\ref{thm:injsmallexs}.  Using \cite{SnapPy, hikmot2016},
    one can show $W$ is hyperbolic with volume
    $\approx 9.67280773079$. The maps $H_1(T_i; \Z) \to H_1(W; \Z)$
    are isomorphisms since the two components of $L$ have linking
    number $1$. Using SnapPy \cite{SnapPy}, one checks that $W$ is the
    orientation cover of the nonorientable census manifold $X = x064$.
    As $\partial X$ is a torus and $H_1(X; \Z) = \Z^2$, the nontrivial
    covering transformation of $W \to X$ must interchange the two
    components of $\partial W$ and act trivially on $H_1(W; \Z)$,
    giving us the needed symmetry of $W$.  Finally, a presentation for
    $\pi(X)$ is
    $\big\langle a, b \ \big| \
    a^2bab^{-2}a^{-1}b^2a^{-1}ba^{-1}b^{-2} = 1\big\rangle$, and
    applying \cite{Brown1987} to compute the BNS invariant shows there
    are many homomorphisms $\pi_1(X; \Z) \to \Z$ with finitely
    generated kernel and hence both $X$ and $W$ fiber over the circle
    by \cite{Stallings1962}.  }
  \label{fig:link}
\end{figure}

\begin{proof}[Proof of Theorem~\ref{thm:injsmallexs} ]
  Our examples are made by Dehn filling a certain \2-cusped hyperbolic
  \3-manifold.  Let $W$ be a compact manifold with
  $\bdry W = T_1 \sqcup T_2$ where both $T_i$ are tori, whose interior
  is hyperbolic, which fibers over the circle, and where maps
  $H_1(T_i; \Z) \to H_1(W; \Z)$ are isomorphisms.  (The last
  condition should be interpreted as saying that $W$ is homologically
  indistinguishable from $T \times I$.)  Further, we require that $W$
  has an orientation-reversing involution that interchanges the
  $T_i$ and acts on $H_1(W; \Z)$ by the identity.  One such $W$ is described
  in Figure~\ref{fig:link}. Since $W$ fibers, there is a
  \1-dimensional face $F$ of the Thurston norm ball so that any
  $\phi \in H^1(W; \Z)$ in the cone $C_F = \R_{>0} \cdot F$ can be
  represented by a fibration.  Choose $\alpha, \beta \in C_F$ that
  form an integral basis for $H^1(W; \Z)$, and let
  $a, b \in H_1(W; \Z)$ be the dual homological basis where
  $\alpha(a) = \beta(b) = 1$ and $\alpha(b) = \beta(a) = 0$.  Let
  $M_n$ be the closed manifold obtained by Dehn filling $W$ along the
  curves in $T_i$ homologically equal to $a - n b$; since $W$ is a
  $\Z$-homology $T \times I$, it follows that $H^1(M_n; \Z)$ is $\Z$
  and is generated by the extension $\phi_n$ of
  $\phitil_n = n \alpha + \beta$ to $M_n$.  We will show there are
  constants $c_1, c_2 > 0$ so that
  \begin{enumerate}[(i)]
  \item \label{item:norm}
    For all $n$, we have 
    $\thnorm{\phi_n} = n \thnorm{\alpha} + \thnorm{\beta} - 2$.
  
  \item \label{item:hyper} For large $n$, the manifold $M_n$ is
    hyperbolic.  Moreover, we have $\vol(M_n) \to \vol(W)$ and
    $\inj(M_n) \sim \frac{c_1}{n^2}$ as $n \to \infty$.
  
   \item \label{item:tube}
     For large $n$, there is a tube $V_n$ in $M_n$ with core geodesic
     $\gamma_n$ of length $2 \inj(M_n)$, with depth 
     $R_n \geq\arcsinh(c_2 n)$, and where $\int_{\gamma_n} \phi_n = 1$. 
  \end{enumerate}
  Here is why these three claims imply the theorem.  From 
  (\ref{item:hyper}) and (\ref{item:tube}), an easy calculation with
  Lemma~\ref{lem:tube} gives a $c_4 > 0$ so that 
  \begin{equation}\label{eq:normgrowth}
  \harnorm{\phi_n} \geq \harnorm{\phi_n |_{V_n}} 
      \geq c_4 n \sqrt{\log n} \mtext{for all large $n$.}
  \end{equation}
  Combining (\ref{eq:normgrowth}) with
  (\ref{item:norm}-\ref{item:tube}) now gives both parts of
  Theorem~\ref{thm:injsmallexs}.  So it remains to prove the claims
  (\ref{item:norm}-\ref{item:tube}).  

  For (\ref{item:norm}), first note that since $\thnorm{\cdotspaced}$
  is linear on $C_F$ we have 
  \[
  \thnorm{\phitil_n} = n \thnorm{\alpha} + \thnorm{\beta}
  \]
  Let $\Stil_n$ be a fiber in the fibration dual to $\phitil_n$, and
  hence $\eulerminus{\Stil_n} = \thnorm{\phitil_n}$ by
  \cite[\S 3]{Thurston1986}.

  For homological reasons, the boundary of $\Stil_n$ has only one
  connected component in each $T_i$.  Thus $\Stil_n$ can be capped off
  with two discs to give a surface $S_n \subset M_n$ which is a fiber
  in a fibration of $M_n$ over the circle.  Hence
  \[
  \thnorm{\phi_n} = \eulerminus{S_n} = \eulerminus{\phitil_n} - 2 
    =  \thnorm{\phitil_n} - 2 = n \thnorm{\alpha} + \thnorm{\beta} - 2
  \]
  proving claim (\ref{item:norm}).

  Starting in on (\ref{item:hyper}), the Hyperbolic Dehn Surgery
  Theorem \cite{ThurstonLectureNotes} shows that $M_n$ is hyperbolic
  for large $n$ and that $\vol(M_n)$ converges to $\vol(W)$. Moreover,
  for large enough $n$, the cores of the Dehn filling solid tori are
  isotopic to the two shortest geodesics in $M_n$.  In fact, these
  core geodesics have the same length: the required involution of $W$
  extends to $M_n$ and thus give an isometry of $M_n$ which
  interchanges them.  Let $\gamma_n \subset M_n$ be the core geodesic
  inside $T_1$.  Using \cite[Proposition~4.3]{NeumannZagier1985} one
  can show that there is a constant $c_1 > 0$ so that
  $\length(\gamma_n) \sim \frac{2 c_1}{n^2}$ as $n \to \infty$.  As
  $2 \inj(M_n) = \length(\gamma_n)$, we have shown (\ref{item:hyper}).

  Finally, for (\ref{item:tube}), first note that 
  $\gamma_n$ meets the fiber surface $S_n$ in a single point, giving
  $\int_{\gamma_n} \phi_n = 1$ if we orient $\gamma_n$ suitably.  So
  it remains only to estimate the depth $R_n$ of the Margulis tube $V_n$
  about $\gamma_n$.  Let $V_n'$ be the image of $V_n$ under the above
  involution of $M_n$.  The picture from the Hyperbolic Dehn Surgery
  Theorem shows that the geometry of $M_n$ outside the
  $V_n \sqcup V_n'$ converges to that of a compact subset of $W$.  As
  $\vol(M_n) \to \vol(W)$, we must have that the volume of $V_n$
  converges as $n \to \infty$.  A straightforward calculation using
  (\ref{eq:voltube}) now gives the lower bound on $R_n$ claimed in
  (\ref{item:tube}). 
\end{proof}

\section{Proof of Theorem~\ref{thm:largenorm}}

\begin{figure}[bt]
\begin{center}
\begin{tikzoverlay}[scale=0.25]{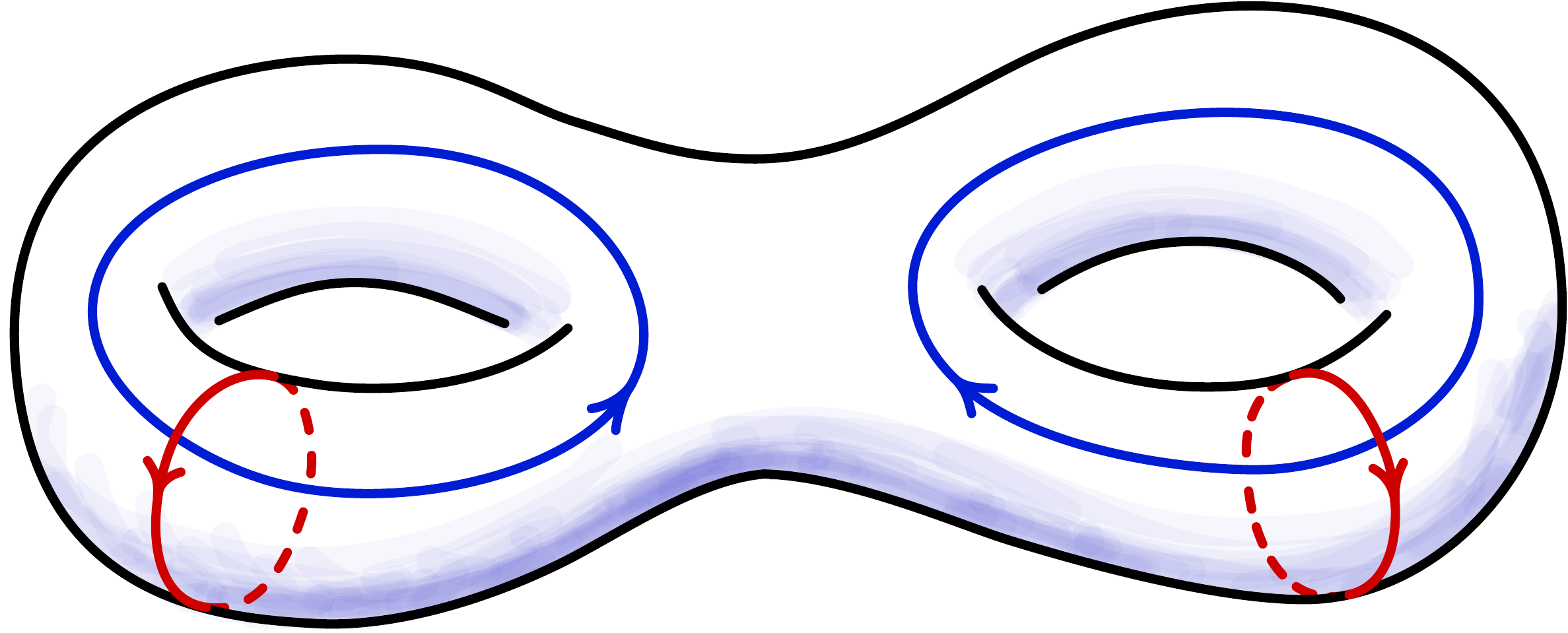}[font=\small]
    \node[right] at (9.5,5.5) {$e_1$};
    \node[right] at (40.5,16.3) {$e_2$};
    \node[left] at (60.7,16.3) {$e_4$};
    \node[left] at (90.0,6.0) {$e_3$};
    \node[below=7pt] at (49.6,1.1) {(a) Generators of $H_1(S; \Z)$};
\end{tikzoverlay} \hspace{1cm}
\begin{tikzoverlay}[scale=0.25]{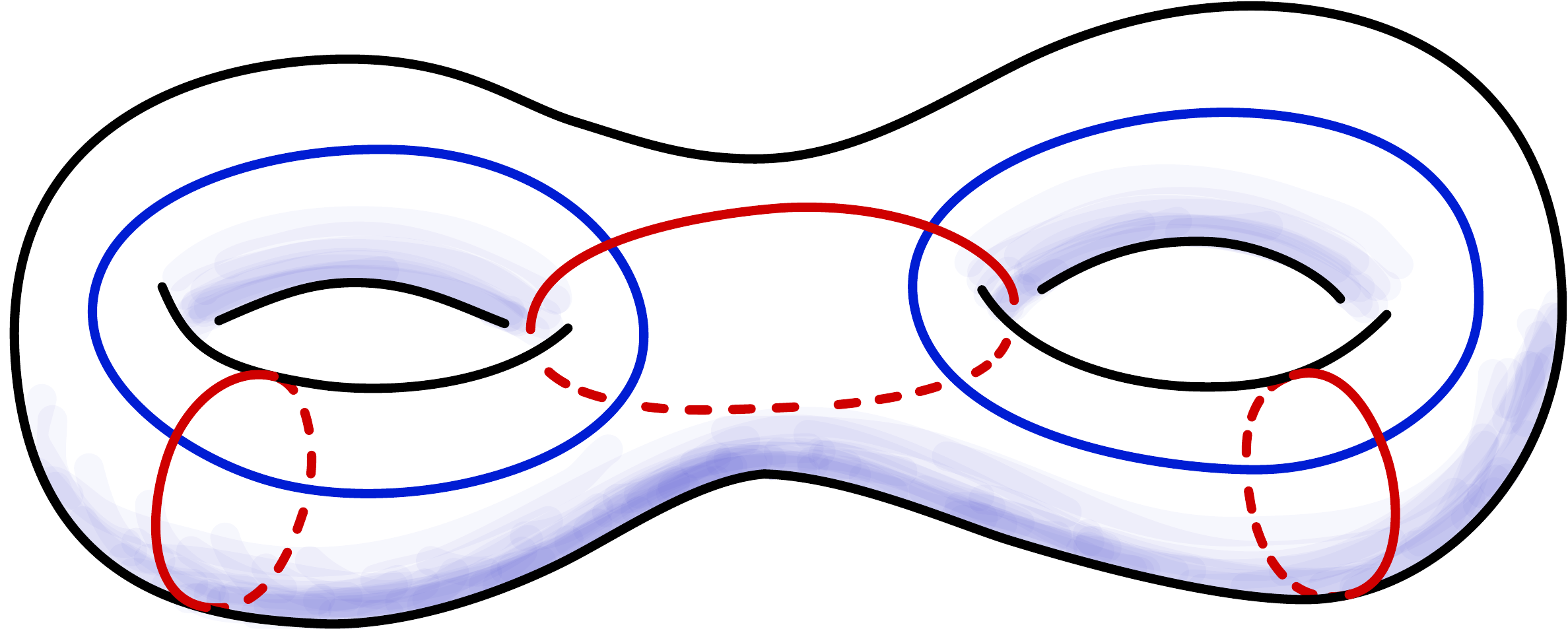}[font=\small]
    \node[right] at (10.1,6.0) {$a$};
    \node[below] at (20.5,30.9) {$b$};
    \node[below] at (49.6,27.0) {$c$};
    \node[below] at (86.8,31.7) {$d$};
    \node[left] at (89.0,7.1) {$e$};
    \node[below=7pt] at (49.6,1.1) {(b) Curves for Dehn twisting};
\end{tikzoverlay}
\end{center}

\vspace{-0.5cm}
  \caption{Conventions for the genus 2 surface $S$.}
  \label{fig:genus2}
\end{figure}

This section is devoted to proving:

\thmlargenorm*

\noindent
We first outline the construction of the $M_n$, sketch why they should
have the desired properties, and finally fill in the details. In
contrast with the rest of this paper, in this section, all
(co)homology groups will use $\Z$ coefficients.

\subsection{The construction}

Let $S$ be a fixed surface of genus 2.  We use the basis $e_i$ for
$H_1(S)$ shown in Figure~\ref{fig:genus2}(a); we take $e^i$ to be the
algebraically dual basis of $H^1(S)$, that is, the one where
$e^i(e_j) = \delta_{ij}$.  Let $W$ be a compact hyperbolic \3-manifold
with totally geodesic boundary, where $\partial W$ has genus 2 and the
map $H_1(\partial W) \to H_1(W)$ is onto; one possible choice for $W$
is the tripus manifold of \cite[\S 3.3.12]{Thurston1997}.  Note that
homologically $W$ is indistinguishable from a genus 2 handlebody.  Let
$W_1$ and $W_2$ be two copies of $W$ whose boundaries have been marked
by $S$ so that
$\ker\left( H_1(S) \to H_1(W_1) \right) = \pair{e_2, e_4}$ and
$\ker\left( H_1(S) \to H_1(W_2) \right) = \pair{e_2, e_3}$.  Thus
$H_1(W_1)$ is spanned by the images of $e_1$ and $e_3$ and so
$H^1(W_1)$ is spanned by natural extensions of $e^1$ and $e^3$.  For 
a carefully chosen pseudo-Anosov $f \maps S \to S$, the examples
used in Theorem~\ref{thm:largenorm} are
\[
M_n = W_1 \cup_{f^n} W_2 = \rightquom{W_1 \sqcup W_2}{(x \in \partial
  W_1)\sim (f^n(x) \in \partial W_2)}{3pt}{\big}
\]
We first sketch the requirements on $f$ and the overall structure of
the argument.  Note that our particular markings mean that
$H^1(M_0) = \Z = \pair{e^1}$, where we are taking $f^0$ to be the
identity map on $S$.  The key requirement is that
$f^* \maps H^1(S) \to H^1(S)$ preserves the subspace $\pair{e^1, e^3}$
and acts on it as an Anosov matrix in $\SL{2}{\Z}$.  We will also
arrange that $H^1(M_n) = \Z = \pair{\phi_n}$ for all $n$.  Since the
action of $f^*$ on $\pair{e^1, e^3}$ is complicated, the coefficients
of $\phi_n$ with respect to $\{e^1, e^3\}$ will grow exponentially in
$n$.  This will force the restriction of $\phi_n$ to $W_1$ to have
Thurston norm which is exponential in $n$, and a standard lemma will
then give that $\phi_n$ itself has Thurston norm which is exponential
in $n$.  We will also show that $\vol(M_n)$ grows (essentially)
linearly in $n$, and combining these will give part
(\ref{item:expgrowth}) of Theorem~\ref{thm:largenorm}.

The specific choice we make for $f$ is given in the next lemma.
\begin{lemma}\label{lem:pA}
  There exists a pseudo-Anosov $f \maps S \to S$ whose action $f_*$ on
  $H_1(S)$ is given by 
  \[
  B = \left(\begin{array}{rrrr}
                3 & 0 & 1 & 0 \\
                1 & 0 & 0 & -1 \\
                -1 & 0 & 0 & 0 \\
                1 & 1 & 1 & 3
              \end{array}\right)
 \]
\end{lemma}
\begin{proof}
  The homomorphism $\MCG(S) \to \Aut(H_1(S))$ induced by taking the
  action of a mapping class on homology is onto the
  symplectic group of the form given by
  \[
  J = \left(\begin{array}{rr|rr}
              0 & -1 & 0 & 0 \\
              1 & 0 & 0 & 0 \\
              \hline
              0 & 0 & 0 & 1 \\
              0 & 0 & -1 & 0
            \end{array}\right)
 \]
 where we are following the conventions given in
 Figure~\ref{fig:genus2}(a).  It is easy to check that $B^t J B = J$,
 and so there is some $f \in \MCG(S)$ with $f_* = B$.  Composing $f$
 with a complicated element of the Torelli group if necessary, we can
 arrange that $f$ is pseudo-Anosov as required.

 Alternatively, consider the following product of Dehn twists:
\[
f = \tau_a \circ \tau_d^{-1} \circ \tau_c \circ \tau_b^{-1} \circ \tau_d \circ \tau_c^{-1} \circ \tau_e^{-1}
\]
where the curve labeling conventions for the right-handed Dehn twists
$\tau$ follow Figure~\ref{fig:genus2}(b).  An easy calculation gives
$f_* = B$, and using \cite{Twister, SnapPy, hikmot2016} one can
rigorously verify that the mapping torus of $f$ is hyperbolic with volume
$\approx 7.51768989647$; in particular, $f$ is pseudo-Anosov.  
\end{proof}

The precise technical statements needed to prove
Theorem~\ref{thm:largenorm} are the following three lemmas.

\begin{lemma}\label{lem:alggengrows}
  For all $n$, the group $H^1(M_n) \cong \Z$ is generated by
  $\phi_n = a_n e^1 + c_n e^3$ where both $a_n$ and $c_n$ grow
  exponentially in $n$ at a rate of
  $\lambda = \frac{3 + \sqrt{5}}{2} \approx 2.62$.
\end{lemma}

\begin{lemma}\label{lem:geomgrows}
  The manifolds $M_n$ are all hyperbolic with injectivity radius
  bounded uniformly below, and $\vol(M_n) \asymp n$ as $n
  \to \infty$. 
\end{lemma}

\begin{lemma}\label{lem:normsplitting}
  Suppose $M^3$ is a closed irreducible \3-manifold.  Suppose
  $F \subset M$ is an incompressible surface dividing $M$ into
  submanifolds $A$ and $B$.  For all $\phi \in H^1(M)$ we have
  \[
  \thnorm{\phi} \geq \thnorm{\phi_A} + \thnorm{\phi_B}
  \]
  where $\phi_A$ and $\phi_B$ are the images of $\phi$ in $H^1(A)$
  and $H^1(B)$ respectively. 
\end{lemma}

We first prove Theorem~\ref{thm:largenorm} assuming the
lemmas, and then establish each of them in turn.

\begin{proof}[Proof of Theorem~\ref{thm:largenorm}]
  By Lemma~\ref{lem:normsplitting}, if $\phibar_n$ denotes the
  restriction of $\phi_n$ to $H^1(W_1)$, we have 
  \[
  \thnorm{\phi_n} \geq \thnorm{\phibar_n}
  \]
  Since $W_1$ is hyperbolic with totally geodesic boundary, it is
  atoroidal and acylindrical and hence the Thurston norm on $H^1(W_1)$
  is nondegenerate.  As any two norms on a finite-dimensional
  vector space are uniformly comparable,
  Lemma~\ref{lem:alggengrows} gives that
  $\thnorm{\phibar_n} \asymp \lambda^n$ as $n \to \infty$.  Since
  $\vol(M_n) \asymp n$ by Lemma~\ref{lem:geomgrows}, we have that
  $\thnorm{\phi_n}$ grows exponentially in $\vol(M_n)$ as required.
\end{proof}

\begin{remark}
  In fact, working a little harder one can make the rate of
  exponential growth explicit, namely
  \begin{equation}\label{eq:explicit}
    \log{\thnorm{\phi_n}}  >  \ 0.348 \cdot \vol(M_n) \mtext{for large $n$.}
  \end{equation}
  Specifically, take $f$ to be the map constructed in the second proof
  of Lemma~\ref{lem:pA}.  If we use \cite{Tian1990} in the manner of
  \cite[Chapter 12]{Namazi2005} to get a refined version of the model
  for $M_n$, it follows that $\vol(M_n)/n$ limits to
  $\vol(M_f) \approx 7.51768989$.  Combining with the explicit formula
  for $\lambda$ in Lemma~\ref{lem:alggengrows} gives
  (\ref{eq:explicit}).
\end{remark}

\begin{proof}[Proof of Lemma~\ref{lem:alggengrows}]
First, we show that $H^1(M_n) = \Z$ for all $n$, and moreover identify
the generator $\phi_n$ in terms of the basis $e^i$ for $H^1(S)$.  (The
stronger claim $H_1(M_n) \cong \Z$ is also true, but we have no need
for this here.)  Let
\[
  F = f^* = B^t = \left(\begin{array}{rrrr}
                          3 & 1 & -1 & 1 \\
                          0 & 0 & 0 & 1 \\
                          1 & 0 & 0 & 1 \\
                          0 & -1 & 0 & 3
                        \end{array}\right)
\]
By Mayer-Vietoris, we have 
\begin{equation}\label{eq:MV}
\begin{split}
  H^1(M_n) &= \image\big(H^1(W_1) \to H^1(S)\big) \cap  (f^*)^n \Big(
  \image\big(H^1(W_2) \to H^1(S)\big)\Big) \\
  &= \pair{e^1, e^3} \cap \pair{ F^n(e^1), F^n(e^4)}
\end{split}
\end{equation}
Notice that $F$ preserves the subspace $\pair{e^1, e^3}$ and acts
there by the matrix 
\[
\Fbar = \left(\begin{array}{rr}
3 & -1 \\
1 & 0
\end{array}\right)
\]
which is an Anosov matrix in $\SL{2}{\Z}$. For $n = 0$, the
intersection in (\ref{eq:MV}) is
\[
\pair{e^1, e^3} \cap \pair{e^1, e^4} = \pair{e^1}.
\]
Hence, for general $n$ the intersection is spanned by $F^n(e^1)$ which is $a_n e^1
+ c_n e^3$ where 
\[
\Fbar^n = \left(\begin{array}{rr}
a_n & b_n \\
c_n & d_n
\end{array}\right)
\]
Thus $H^1(M_n) = \Z$ as claimed, with generator $\phi_n$ restricting
to $H^1(W_1)$ as $a_n e^1 + c_n e^3$, where $a_n$ and $c_n$ grow
exponentially in $n$, specifically at a rate
$\lambda = \frac{3 + \sqrt{5}}{2} \approx 2.618034$.
\end{proof}

\begin{proof}[Proof of Lemma~\ref{lem:geomgrows}]
  Though a geometric limit argument could be given to verify the
  injectivity radius and volume claims, we refer for efficiency to the
  main result of \cite{BrockMinskyNamaziSouto2016} for a bi-Lipschitz
  model for the manifolds $M_n$. Following the terminology of
  \cite{BrockMinskyNamaziSouto2016}, the \emph{decorated manifolds}
  $\cM$ are the pair of acylindrical manifolds $W_1$ and $W_2$ with
  totally geodesic boundary, the \emph{decorations} $\mu_1$ and
  $\mu_2$ on the boundaries $\bdry W_1$ and $\bdry W_2$ can be taken
  to be bounded length markings in the induced hyperbolic structure on
  each boundary component, and the gluing map is $f^n$ as above.  The
  \emph{large heights} condition in Theorem~8.1 of
  \cite{BrockMinskyNamaziSouto2016} is evidently satisfied since the
  curve complex distance
  \[
  d_{\cC(\bdry W_2)}\big(f^n(\mu_1),\mu_2\big)
  \]
  grows linearly with $n$ and likewise for
  $\big(\mu_1, f^n(\mu_2)\big)$ in $\cC(\bdry W_1)$.  Furthermore, the
  pair $\big(f^n(\mu_1) , \mu_2\big)$ has \emph{$R$-bounded
    combinatorics}, where $R$ is independent of $n$, for the following
  reason. Picking points $X$ and $Y$ in the Teichm\"uller space of $S$
  where $\mu_1$ and $\mu_2$ have bounded length, the quasi-fuchsian
  manifolds $Q\big(f^n(Y),X\big)$ converge strongly to a manifold $N$
  with injectivity radius bounded below by
  \cite[Corollary~3.13]{McMullen1996}.  Thus there is a uniform lower
  bound on the injectivity radii of the approximates $Q(f^n(Y),X)$.
  Applying the Bounded Geometry Theorem \cite{Minsky2001} to the
  $Q(f^n(Y), X)$ gives the needed uniform upper bound on the distance
  between $f^n(\mu_1)$ and $\mu_2$ in any subsurface projection.

  The \emph{model manifold} for the $(\cM,R)$-gluing $X_n$ determined
  by these data is as follows.  Let $\Mtil_f$ be the fiber cover of
  the mapping torus $M_f$ and let $\Mtil_f[0,1]$ be a fundamental
  domain for the action of $f$ as an isometric covering translation
  $\alpha_f: \Mtil_f \to \Mtil_f$ bounded by a choice of fiber and its
  translate by $\alpha_f$.  Defining $\Mtil_f[k,k+1]$ to be
  $\alpha_f^k (\Mtil_f[0,1])$, we use $\Mtil_f [0,n]$ to denote the
  union of $n$ successive such fundamental domains.  Then the model
  manifold $\mathbb{M}_{X_n}$ is the gluing of $W_1$ and $W_2$ along
  their boundary to $\Mtil_f [0,n]$ in the manner described in
  \cite[\S 2.15]{BrockMinskyNamaziSouto2016}.  Given that $\Mtil_f$ is
  periodic, we know that $\inj(\mathbb{M}_{X_n})$ is bounded below
  independent of $n$ and that
  $\vol(\mathbb{M}_{X_n}) \sim \vol(M_f) \cdot n$ as $n \to \infty$.

  Now Theorem~8.1 of \cite{BrockMinskyNamaziSouto2016} gives a $K$ so
  that for all large $n$ there is a $K$--bi-Lipschitz diffeomorphism
  \[
  f_{X_n} \colon \mathbb{M}_{X_n} \to M_n
  \]
  Combined with the above facts about the geometry of 
  $\mathbb{M}_{X_n}$, this gives the claimed properties for $M_n$ and
  so proves the lemma.
\end{proof}

\begin{proof}[Proof of Lemma~\ref{lem:normsplitting}]
  Let $S \subset M$ be a surface dual to $\phi$ which is \emph{taut},
  that is, the surface $S$ realizes $\thnorm{\phi}$, is
  incompressible, and no union of components of $S$ is separating.  As
  $F$ and $S$ are incompressible, we can isotope $S$ so that
  $F \cap S$ consists of curves that are essential in both $S$ and
  $F$; in particular, every component of $S \setminus F$ has
  non-positive Euler characteristic.  As $S \cap A$ and $S \cap B$ are
  dual to $\phi_A$ and $\phi_B$ respectively, we have
  \[
  \thnorm{\phi} = -\chi(S) = -\chi(S \cap A) - \chi(S \cap B) =
  \eulerminus{S \cap A} + \eulerminus{S \cap B} \geq \thnorm{\phi_A}
  + \thnorm{\phi_B}
  \]
  as desired. 
\end{proof}

{\RaggedRight 
\bibliographystyle{math} 
\bibliography{big-IHS}
}

\end{document}